\documentclass[12pt]{article}
\usepackage{eurosym}
\usepackage{amssymb}
\usepackage[english]{babel}
\usepackage{amsmath,amsfonts,amssymb,amsthm}
\usepackage{multirow} 
\usepackage{booktabs} 
\usepackage{pstricks,pst-node,pst-tree}
\usepackage{cite,amsmath,amssymb,amsthm,amsfonts}
\usepackage[top=2.5cm,bottom=2.5cm,right=2.5cm,left=2.5cm,headsep=0.1in,includehead]{geometry}

\usepackage[authoryear, sort&compress]{natbib}
\usepackage{hyperref}

\newtheorem{thm}{\sc Theorem}[section]

\newtheorem{prop}{\sc Proposition}[section]

\newtheorem{ex}{\sc Example}[section]

\title{Cost allocations in interval inventory situations: the SOC and Shapley approaches}
\author{J.C. Gonçalves-Dosantos\footnote{Corresponding author: jgoncalves@umh.es}, A. Meca, I. Ozcan\\\\Center of Operations Research, University Miguel Hernandez, Elche, Spain.}
\date{\today }

\begin{document}

\maketitle

\begin{abstract}  
\noindent Uncertainty in demand and supply conditions poses critical challenges to effective inventory management, especially in collaborative environments. Traditional inventory models, such as those based on the Economic Order Quantity (EOQ), often rely on fixed parameters and deterministic assumptions, limiting their ability to capture the complexity of real-world scenarios. This paper focuses on interval inventory situations, an extension of classical models in which demand is represented as intervals to account for uncertainty. This framework allows for a more flexible and realistic analysis of inventory decisions and cost-sharing among cooperating agents. We examine two interval-based allocation rules, the interval SOC-rule and the interval Shapley rule, designed to distribute joint ordering costs fairly and efficiently under uncertain demand. Their theoretical properties are analyzed, and their practical applicability is demonstrated through a case study involving the coordination of perfume inventories across seven Spanish airports, based on 2023 passenger traffic data provided by AENA. The findings highlight the potential of interval-based models to enable a robust and equitable allocation of inventory costs in the face of operational uncertainty.
\end{abstract}



\section{Introduction} 

\noindent Cooperative models offer a rigorous and versatile framework for analyzing how multiple agents, be they individuals, organizations, or institutions, can collaborate to achieve mutually beneficial outcomes \cite{Myerson1991}. This is particularly relevant in contexts where the distribution of shared resources, costs, or risks must be negotiated. Traditionally, cooperative games have been employed under the assumption of deterministic environments, where precise data on costs and payoffs is available \cite{Shapley1953}. However, in an increasingly volatile world characterized by fluctuating demand, disrupted supply chains, and pervasive economic uncertainty, such assumptions often prove inadequate \cite{Snyder2006}.

\noindent In light of these challenges, interval cooperative games have emerged as a powerful extension of classical models \cite{Branzei2010Survey}. By representing payoffs as intervals rather than fixed numerical values, interval cooperative games provide a structured yet flexible mechanism for capturing uncertainty and variability in real-world decision-making. This interval-based approach allows for more robust modeling of ambiguity in key parameters \cite{Moore2003}, thereby enhancing the capacity to identify fair and stable solutions even when exact data are unavailable or imprecise. As such, interval cooperative games serve as a valuable tool for bridging the gap between theoretical modeling and the complexities of practical decision environments.

\noindent Among the many domains where cooperative strategies are essential, inventory management stands out due to its potential for cost optimization through collaboration. In particular, cooperative inventory models, often grounded in the classical Economic Order Quantity (EOQ) framework \cite{Hillier2015}, have demonstrated that joint ordering among agents can significantly reduce overall costs by exploiting economies of scale \cite{Meca2003,Meca2004}. A comprehensive survey on cooperative game theory and its applications to inventory management is provided in \cite{Fiestras2011}. Nevertheless, a major limitation of these traditional models lies in their reliance on deterministic demand assumptions, assumptions that are frequently violated in practice due to seasonal variability, shifting market trends, and unexpected disruptions.

\noindent To address this limitation, interval inventory games are developed as an extension of cooperative inventory models. These games incorporate demand uncertainty directly by modeling demand as intervals, thereby enabling a more realistic analysis of cooperative behavior and cost-sharing mechanisms under uncertainty. This advancement not only enriches the theoretical underpinnings of cooperative game theory but also enhances its applicability to dynamic, real-world environments, thus contributing to the broader field of operations research.

\noindent Despite increasing interest in integrated inventory management and the extensive application of game theory to supply chain coordination, there remains a lack of cooperative game-theoretic models explicitly addressing inventory systems subject to demand uncertainty. This deficiency arises from several challenges. Primarily, the analytical complexity introduced by uncertain demand—often modeled via stochastic, fuzzy, or grey frameworks—hampers the tractability of cooperative game formulations. Additionally, cooperative game theory requires stable allocation mechanisms, such as core solutions or the Shapley value, whose computation and existence become problematic when payoffs depend on probabilistic or imprecise demand parameters. To the best of our knowledge, only a limited number of studies have attempted to fill this gap. Notably, \cite{Donmez2024} propose a model based on equal surplus sharing within grey inventory games, underscoring both the potential and the current paucity of research at the intersection of cooperative game theory and uncertain inventory systems. Robust supply chain design under demand uncertainty has been explored by \cite{Gomez2024}, who demonstrate that a Shapley-based cost allocation mechanism outperforms alternative cooperative approaches. Furthermore, foundational work in cooperative inventory management employs the Shapley value to centralize inventories and equitably share joint benefits, as detailed in \cite{Dror2002}. The application of interval-valued demand representations in multilevel supply chains has been examined by \cite{Bai2017}, providing a fundamental framework for incorporating uncertainty into inventory models similar to those considered herein. Additionally, \cite{Kimms2009} introduce an algorithmic approach for cooperative interval-valued games, enabling cost allocation under uncertainty through the computation of interval core solutions in lot-sizing contexts. Complementing these contributions, \cite{Alparslan2019Hurwicz} investigate how the Hurwicz criterion informs egalitarian outcomes in cooperative interval games, offering theoretical insights relevant to the fairness and robustness of interval-based allocation rules.

\noindent In this paper, we focus on the study of interval inventory situations, with a particular emphasis on interval inventory models derived from EOQ models under uncertain demand. We investigate two central allocation rules, the interval SOC-rule and the interval Shapley rule, both designed to fairly and efficiently distribute the total joint inventory costs among cooperating agents. These rules are examined in terms of their theoretical properties, including efficiency, fairness, and stability. To illustrate the practical relevance of our approach, we apply the proposed models to a real-world case study involving the cooperative management of perfume inventories across seven major Spanish airports. Using 2023 passenger traffic data provided by AENA (Aeropuertos Españoles y Navegación Aérea) \cite{aena2023}, we evaluate the implications of interval uncertainty on inventory coordination and cost allocation. Through this application, we aim to provide both theoretical insights and practical tools for improving cooperative strategies in inventory management and related domains facing uncertain conditions.

\noindent The structure of this paper is organized as follows. Section \ref{sec:2} introduces the foundational concepts necessary for the subsequent analysis. It begins with a general overview of cooperative cost games, followed by a review of cooperative inventory games. It then presents the basic definitions and notations related to intervals, and concludes with a description of cooperative interval games. Section \ref{sec:3} extends classical inventory models to an interval-based framework. Within this section, we define and characterize the interval SOC-rule as applied to interval inventory situations. Section \ref{sec:4} introduces and characterizes the interval Shapley rule for these same settings. Section \ref{sec:5} finalizes the examination of cost allocation in interval inventory scenarios by conducting a comparative analysis of the interval individual cost, the interval Shapley rule, and the interval SOC-rule. This analysis is illustrated through a case study involving perfume inventory management for duty-free shops located in major Spanish airports. Finally, Section \ref{sec:6} presents the conclusions and outlines potential directions for future research.

\section{Preliminaries on Cooperative Models}\label{sec:2}  

\noindent In this section we review the main models on which our inventory situations with intervals are based as well as the main concepts about intervals and the notation we adopt throughout the paper. We begin by presenting preliminary concepts related to cooperative cost games. 

\subsection{Cooperative cost games}\label{CCG} 

\noindent A cooperative cost game with transferable utility (TU cost game) is a formal model in game theory in which a finite set of players can reduce their cost by forming coalitions, and the total cost generated by each coalition can be freely distributed among its members. In such games, the central focus lies on how to allocate the total costs from cooperation in a fair or stable manner.

\noindent Formally, a TU cost game is defined by a pair $(N,c)$ where $N=\{1,2,...,n\}$ is the set of players and $c:2^{N}\rightarrow \mathbb{R}$ is the characteristic function, which assigns to every coalition $S\subset N$ a real number $c(S)$, representing the total cost that the members of $S$ can achieve by cooperating, being  $c(\emptyset )=0$. We denote the size of a coalition $S\subset N$ by $|S|$ and the family of TU cost games with player set $N$ by $G^{N}$. 

\noindent It is well known that $G^{N}$ is a $(2^{|N|}-1)$-dimensional linear space for which unanimity games form an interesting basis. Indeed, given a coalition $S\in 2^{N}\setminus \{\emptyset \}$, the unanimity game based
on $S$, $u_{S}:2^{N}\rightarrow \mathbb{R}$, is defined as follows,
\begin{equation*}
u_{S}(T)=%
\begin{cases}
1, & \text{if }S\subseteq T, \\ 
0, & \text{otherwise}.%
\end{cases}%
\end{equation*}%
\noindent We may notice that every TU-game $(N,v)\in G^{N}$ can be written in a
unique way as a linear combination of the unanimity games
\begin{equation*}
c=\sum_{T\in 2^{N}\setminus \{\emptyset \}}c_{T}u_{T}\quad \text{with}\quad
c_{T}=\sum_{S:S\subseteq T}(-1)^{|T|-|S|}c(S).
\end{equation*}

\noindent The transferable utility assumption implies that the cost $c(S)$ can be redistributed among the members of $S$ in any way they agree. Solutions to such games often involve concepts like the core and the Shapley value, which provide frameworks for a fair and stable distribution of cost. This model is particularly useful in economic scenarios where agents can form binding agreements and where monetary or utility transfers are possible, such as in cost-sharing problems and supply chain collaborations. 

\noindent The \textit{core} of a TU cost game is the set of feasible and coalitionally stable allocations. It is defined as:
\[
\text{C}(N,c) = \left\{ x \in \mathbb{R}^n \,\middle|\, \sum_{i \in N} x_i = c(N), \quad \sum_{i \in S} x_i \leq c(S) \text{ for all } S \subseteq N \right\}.
\]

\noindent The core is defined as a convex polyhedron that can take one of three possible forms: it may be small (comprising a single allocation), extensive (concave games), or even empty. Games in which the core is non-empty are referred to as balanced games.

\noindent The \textit{Shapley value} is a single-valued solution concept that distributes utility fairly among players based on their marginal contributions. It assigns to each player \( i \in N \) the average marginal contribution across all possible orderings of the players. It is defined by:
\[
\phi_i(N,c) = \sum_{S \subseteq N \setminus \{i\}} \frac{|S|! \cdot (n - |S| - 1)!}{n!} \left[ c(S \cup \{i\}) - c(S) \right], \quad \forall i \in N.
\]

\noindent In his seminal 1953 paper, see \cite{Shapley1953}, Shapley provided an axiomatic characterization of this value for TU games. The characterization is based on four fundamental axioms that reflect desirable properties of efficiency, symmetry, dummy player, and additivity. Since its introduction, the Shapley value has proved to be one of the most important rules in cooperative games (see, for example, \cite{flores2019evaluating}).

\noindent Concave games represent a significant class of TU cost games in which the incentives for cooperation are particularly strong. A cost game $(N,c)$ is said to be \textit{concave} (or submodular) if for all coalitions $S, T \subseteq N$, the following inequality holds, $c(S)+ c(T) \geq c(S \cup T) + c(S \cap T)$. An equivalent condition, known as the \textit{increasing marginal returns property}, states that for all  $S \subseteq T \subseteq N$ and $i \in N \setminus T$, $c(S \cup \{i\}) - c(S) \geq c(T \cup \{i\}) - c(T).$ This condition implies that a player's marginal cost contribution diminishes as they are added to increasingly larger coalitions. In other words, the value of a player is greater when they join larger coalitions.

\noindent The concavity of a cooperative game has several important implications. Firstly, concave games are balanced, which ensures that the core is always non-empty. This provides strong guarantees for the stability of cooperative arrangements. Furthermore, in concave games, the Shapley value always lies within the core. This makes concave games especially relevant in applications where fairness and stability are both essential. Concave games are commonly encountered in economics and operations research, particularly in settings involving economies of scale, resource allocation, and cost sharing, where joint cooperation yields increasing collective benefits relative to the size of the group.

\noindent Subsequently, we examine a specific subclass of these games—inventory games—as initially introduced by \cite{Meca2003,Meca2004}.

\subsection{Inventory Games} \label{IG}

\noindent Inventory games represent a class of cooperative cost games emerging from joint inventory management scenarios. They capture situations in which a group of agents, such as firms, retailers, or regional warehouses, can reduce overall inventory-related costs by coordinating their ordering and storage activities. The central premise is that cooperation enables cost reductions through mechanisms such as risk pooling, economies of scale, and the utilization of shared resources.

\noindent These games are typically modeled within the framework of the Economic Order Quantity (EOQ) model, a fundamental tool in operations management for systems with deterministic and constant demand. By aligning their replenishment strategies, agents can achieve cost efficiencies through joint procurement and centralized inventory management. The EOQ-based formulation provides a structured foundation for analyzing the benefits of cooperation and the fair allocation of the resulting savings.

\noindent In the classical EOQ setting, each agent (e.g., a retailer or production facility) faces a continuous, known demand rate, $d_{i}\geq 0$ (units per time), and must decide how frequently and in what quantity to replenish inventory. The total cost for each agent typically consists of two components: (1) ordering cost ($a>0$), which is a fixed cost incurred every time an order is placed, regardless of the order size; (2) holding cost $\left( h_{i}>0\right)$, which is proportional to the average inventory level and reflects the cost of storing goods over time.

\noindent For an individual agent $i$, the average inventory cost per unit time is a function of the order size $Q_{i},$ which is given by 
\begin{equation*}
c\left( Q_{i}\right) =a\dfrac{d_{i}}{Q_{i}}+h_{i}\dfrac{Q_{i}}{2}
\end{equation*}%
and the optimal order size $Q^*_{i}$ is equals to $\sqrt{2ad_{i}/h_{i}}.$
Thus, the optimal inventory cost per unit of time is 
\begin{equation*}
c(Q^*_{i})=\sqrt{2ad_{i}h_{i}}=2am_{i},
\end{equation*}%
where $m_{i}=\dfrac{d_{i}}{Q^*_{i}}$ is the optimal number of orders
per unit of time.

\noindent In cooperative settings, agents with similar inventory needs may form coalitions to place joint orders, reducing order frequency and benefiting from economies of scale. By centralizing procurement, they can share ordering costs based on their relative usage or demand rates. The coalition’s cost is typically modeled using the EOQ formula applied to aggregate demand. This forms a cost game, where the key challenge is to allocate total costs fairly and in an incentive-compatible manner. EOQ-based inventory games are concave, ensuring stable allocations and enabling solution concepts like the Shapley value. These models are particularly relevant in supply chain coordination, group purchasing, and collaborative inventory management.

\noindent Formally, an inventory situation is a triplet $\left( N,a,\left\{
m_{i}\right\} _{i\in N}\right) ,$ where $N=\left\{ 1,...,n\right\} $ i is the group of agents agreeing to place joint orders of a certain good, $a$
is the fixed ordering cost and $\{m_{1},\ldots ,m_{n}\}$ is the set of optimal
numbers of orders for the agents. 

We assume that the agents have decided to place joint orders in \( N \)  to save their
inventory costs. This means that, if we denote by \( Q_i \) the order size of  agent \( i \), it holds that 
\[
\frac{Q_i}{d_i} = \frac{Q_j}{d_j}, \quad \text{for all } i, j \in N.
\]

The average inventory cost per unit time for a given coalition $S \subseteq N$ is given by 
\begin{equation*}
c(\{Q_{j}\}_{j \in S})=\dfrac{ad_{i}}{Q_{i}}+\sum_{j\in S}\dfrac{h_{j}Q_{j}}{2}.
\end{equation*}

\noindent Now, since $Q_{j}=\dfrac{d_j}{d_i}Q_i$, the average inventory cost per unit time can be seen as a function of the order size $Q_{i}$ given by 
\begin{equation*}
c(Q_{i})=\dfrac{ad_{i}}{Q_{i}}+\dfrac{Q_{i}}{2d_{i}}\sum_{j\in S}h_{j}d_{j}.
\end{equation*}

The optimal order size is equal to 
\begin{equation*}
\hat {Q_{i}}=\sqrt{\dfrac{2ad_{i}^{2}}{\sum\nolimits_{j\in S}d_{j}h_{j}}},%
\text{ for all }i\in S.
\end{equation*}%

\noindent The minimal total inventory cost per unit time is given by 

\begin{equation*}
c(\hat{Q_{i}})=\dfrac{ad_{i}}{\hat{Q_{i}}}+\dfrac{\hat{Q_{i}}}{2d_{i}}\sum_{j\in S}h_{j}d_{j}=2a\sqrt{\sum_{j\in S}m_{j}^{2}}.
\end{equation*}

\noindent The inventory game arising from the inventory situation $(N,a,\left\{
m_{i}\right\} _{i\in N})$ is denoted by $\left( N,c\right) $. Here, $c(S)$
is the minimum average inventory cost per unit time for the agents in $S$ if they
place their orders jointly, and defined by 

\begin{equation*}
c(S)=2a\sqrt{\sum\nolimits_{i\in
S}m_{i}^{2}}  
\end{equation*}
for each $S\subseteq N$, and $c(\emptyset)=0.$

In \cite{Meca2003}, it is proven that inventory games are concave. The authors propose a proportional core allocation that can be achieved through a population monotonic allocation scheme, known as the SOC-rule. Formally, the SOC-rule for an inventory game $\left( N,c\right)$ is defined as follows 
\begin{equation} \label{soc}
\sigma_i(N,c):=\frac{m_i^2}{
\sum_{j\in N}m_{j}^{2}}c(N).
\end{equation}

A characterization of the SOC-rule based on a transfer property is presented in \cite{Meca2004}. Specifically, the SOC-rule is shown to be the unique efficient allocation rule for inventory cost games that satisfies the properties of transfer and null player.

The SOC-rule can also be defined in terms of the inventory situation. That is, in terms of the parameters that describe it, as follows  

\begin{equation} \label{soc}
\sigma_i\left( N,a,\left\{m_{j}\right\} _{j\in N}\right):=\frac{2am_i^2}{\sqrt{\sum_{j\in N}m_{j}^{2}}}.
\end{equation}
 

\noindent Next, we outline essential notions concerning intervals and establish the notation used throughout the paper. 

\subsection{Intervals and their operators}
\noindent We denote by $\mathcal{I}(\mathbb{R})$ the set of all closed
intervals in $\mathbb{R}$, by $\mathcal{I}(\mathbb{R}^{+})$ the set of all
closed intervals in $\mathbb{R}^{+}$, and by $\mathcal{I}(\mathbb{R}%
^{+})^{N} $ the set of all $n$-dimensional vectors with components in $%
\mathcal{I}(\mathbb{R}^{+})$.

\noindent Let $a,b\in \mathcal{I}(\mathbb{R})$ with $a=[\underline{a},%
\overline{a}]$, $b=[\underline{b},\overline{b}]$, $|a|=\overline{a}-%
\underline{a}$ and $\beta \in \mathbb{R}^{+}$, then  
\begin{equation*}
a+b=[\underline{a}+\underline{b},\overline{a}+\overline{b}],\quad \beta
a=[\beta \underline{a},\beta \overline{a}].
\end{equation*}%
\noindent The subtraction operator $a-b$ is defined, only if $|a|\geq |b|$,
by 
\begin{equation*}
a-b=[\underline{a}-\underline{b},\overline{a}-\overline{b}].
\end{equation*}%
\noindent Let $a,b\in \mathcal{I}(\mathbb{R}^{+})$, then  
\begin{equation*}
a\cdot b=[\underline{a}\underline{b},\overline{a}\overline{b}].
\end{equation*}%
\noindent Let $a\in \mathcal{I}(\mathbb{R}^{+})$, then  
\begin{equation*}
\sqrt{a}=[\sqrt{\underline{a}},\sqrt{\overline{a}}].
\end{equation*}%
\noindent The division operator is defined, only if $\underline{a}\overline{b%
}\leq \overline{a}\underline{b}$ and $\underline{b},\overline{b}\neq 0$, by 
\begin{equation}
\frac{a}{b}=\left[ \frac{\underline{a}}{\underline{b}},\frac{\overline{a}}{%
\overline{b}}\right] .  \tag{$2.1$}
\end{equation}%
\noindent We say that $a$ is weakly better than $b$, which we denote by $%
a\succeq b$, if and only if $\underline{a}\geq \underline{b}$ and $\overline{%
a}\geq \overline{b}$. We also use the reverse notation $a\preceq b$, if and
only if $\underline{b}\geq \underline{a}$ and $\overline{b}\geq \overline{a}$.

\noindent Finally, we complete this section of preliminaries with a review of the cooperative interval games. For additional details, the reader is referred to \cite{Alparslan2009,Alparslan2010Axiomatization,Alparslan2010Book,Alparslan2011SetValued,Alparslan2013Sequencing,Alparslan2014Mountain,Branzei2011Convexity}

\subsection{Cooperative Interval Games} \label{CIG}

\noindent The model of interval cooperative games constitutes a generalization of the classical transferable utility (TU) game framework described above in the section \ref{CCG}. In this extended model, uncertainty in the worth of coalitions is represented by closed intervals rather than precise real values, allowing for a more flexible and robust analysis of cooperative situations under imprecise or incomplete information.

\noindent Formally, a cooperative interval game in coalitional form is an ordered pair 
$(N,w)$ where $N$ is the set of players, and $%
w:2^{N}\rightarrow \mathcal{I}(\mathbb{R})$ is the characteristic function
such that $w(\emptyset )=[0,0]$. For each $S\in 2^{N}$, the worth $w(S)$ of
the coalition $S$ in the interval game $(N,w)$ is of the form $[\underline{w}%
(S),\overline{w}(S)]$, where $\underline{w}(S)$ is the lower bound and $%
\overline{w}(S)$ is the upper bound of $w(S)$. We denote by $IG^{N}$ the
family of all cooperative interval games with player set $N$.

\noindent Let $S\in N$, $S \neq \emptyset$, $I\in \mathcal{I}(\mathbb{R})$, and let $u_{S}$ be the unanimity game based on $S$. The
cooperative interval unanimity game $(N,Iu_{S})$ is defined by  

\begin{equation*}
(Iu_{S})(T)=u_{S}(T)I\quad \text{for each } T \subseteq N, T \neq \emptyset.
\end{equation*}

\noindent In As in cooperative cost games, every cooperative interval game $(N,w)\in IG^{N}$
can be written in a unique way as a linear combination of the interval unanimity games 
\begin{equation*}
w=\sum_{T\in 2^{N}\setminus \{\emptyset \}}C_{T}u_{T},\quad \text{where}%
\quad C_{T}=\sum_{S:S\subseteq T}(-1)^{|T|-|S|}w(S).
\end{equation*}%
Here, $C_{T}=[\underline{C_{T}},\overline{C_{T}}]$, with $\underline{C_{T}}%
=\sum_{S:S\subseteq T}(-1)^{|T|-|S|}\underline{w}(S)$ and $\overline{C_{T}}%
=\sum_{S:S\subseteq T}(-1)^{|T|-|S|}\overline{w}(S)$. 



\noindent Some classical TU-games associated with an interval game $w\in
IG^{N}$ play a key role, namely the border games $(N,\underline{w})$, $(N,%
\overline{w})$ and the length game $(N,|w|)$, where $|w|(S)=|w(S)|=\overline{%
w}(S)-\underline{w}(S)$ for each $S\in 2^{N}$.

\noindent For $w_{1},w_{2}\in IG^{N}$ with $|w_{1}(S)|\geq |w_{2}(S)|$ for
each $S\in 2^{N}$, $(N,w_{1}-w_{2})$ is defined by 
\begin{equation*}
(w_{1}-w_{2})(S)=w_{1}(S)-w_{2}(S).
\end{equation*}%
\noindent We call a cooperative interval game $(N,w)$ \textit{concave }$%
\left( submodular\right) $, if $\forall $ $S,T\in 2^{N}$ and $i\in N$ such that 
$S\subset T\subset N\setminus \{i\}$, we have 
\begin{equation*}
w(S\cup \{i\})-w(S)\succeq w(T\cup \{i\})-w(T).
\end{equation*}

The core of cooperative interval games can be defined analogously to that of classic cooperative games. Despite the presence of interval uncertainty, the core retains its interpretation as the set of stable allocations from which no coalition has an incentive to deviate, with suitable adaptations to account for the imprecise valuation of coalitions.

Precisely, the interval core $C(N,w)$ for $(N,w)\in IG^{N}$ is defined by  
\begin{equation*}
C(N,w)=\left\{ (I_{1},\dots ,I_{n})\in \mathcal{I}(\mathbb{R}%
)^{N}\;|\;\sum_{i\in N}I_{i}=w(N),\;w(S)\succeq \sum_{i\in S}I_{i},\;\forall
S \subseteq N, S \neq \emptyset\right\}
\end{equation*}

It is well known that the core of convex/concave interval games is nonempty. However, the Shapley value of cooperative interval games is only well defined for subclass of these called size-monotonic cooperative interval games. 

Explicitly, a cooperative interval game $(N,w)$ is \textit{size-monotonic} if $%
(N,|w|)$ is monotonic, i.e. $|w|(S)\leq |w|(T)$ for all $S,T\in 2^{N}$ with $%
S\subset T$. We denote by $SMIG^{N}$ the class of size
monotonic interval games with player set $N$. For a size-monotonic
cooperative interval game $(N,w)$, $w(T)-w(S)$ is well-defined for all $%
S,T\in 2^{N}$  with $S\subset T$.   

\noindent Denote by $\mathcal{P}(N)$ the set of permutations $\sigma
:N\rightarrow N$ of $N$. The interval marginal operator $m^{\sigma
}:SMIG^{N}\rightarrow I(\mathbb{R})^{N}$ corresponding to $\sigma $,
associates with each $w\in SMIG^{N}$ the interval marginal vector $m^{\sigma
}(w)$ of $w$ with respect to $\sigma $ defined by $m_{i}^{\sigma
}(w)=w(P^{\sigma }(i)\cup \{i\})-w(P^{\sigma }(i)),$ for each $i\in N$,
where $P^{\sigma }(i):=\{r\in N\mid \sigma ^{-1}(r)<\sigma ^{-1}(i)\},$ and $%
\sigma ^{-1}(i)$ denotes the entrance number of player $i$. Now, we notice
that for each $w\in SMIG^{N}$ the interval marginal vectors $m^{\sigma }(w)$
are well-defined for each $\sigma \in \mathcal{P}(N)$, because the
monotonicity of $|w|$ implies $\overline{w}(S\cup \{i\})-\underline{w}(S\cup
\{i\})\geq \overline{w}(S)-\underline{w}(S),$ which can be rewritten as  $%
\overline{w}(S\cup \{i\})-\overline{w}(S)\geq \underline{w}(S\cup \{i\})-%
\underline{w}(S).$ Thus, $w(S\cup \{i\})-w(S)$ is defined for each $S\subset
N$ with $i\notin S$.

\noindent The interval Shapley value $\varphi :SMIG^{N}\rightarrow I(\mathbb{%
R})^{N}$ is defined by  
\begin{equation*}
\varphi \left( N,w\right) :=\frac{1}{n!}\sum_{\sigma \in \mathcal{P}%
(N)}m^{\sigma }(w),
\end{equation*}%
for each $w\in SMIG^{N}.$ We can write this as  
\begin{equation*}
\varphi _{i}\left( N,w\right) =\frac{1}{n!}\sum_{\sigma \in \mathcal{P}%
(N)}\left( w(P^{\sigma }(i)\cup \{i\})-w(P^{\sigma }(i))\right) .
\end{equation*}%

A recent and detailed axiomatic analysis of the interval Shapley value is presented in \cite{Ishihara2023}, where the authors explore its foundational properties and provide a rigorous characterization within the framework of cooperative interval games.

\noindent We are now ready to extend inventory situations to interval settings, confirm the preservation of some of its properties, and define and analyze the SOC-rule in models with interval demand.

\section{Interval inventory situations and the SOC-rule}\label{sec:3}  

\noindent\textit{The Interval Economic Order Quantity (henceforth IEOQ)} situation is a model that is constructed by an inventory situation, as described in section \ref{IG}, under interval uncertainty demand. In other words, the demand of each agent $i$ that must be met is now an interval $d_{i} =\left[ \underline{d_{i}},\overline{d_{i}}\right]$ with $\underline{d_{i}},\overline{d_{i}}\geq 0.$ The ordering cost $a>0$ and the holding costs $h_{i}>0, i \in N$ are deterministic. The average interval inventory cost per unit time is a function of the order size $Q_{i} \geq 0,$ and it is given by %
\begin{equation*}
c\left( Q_{i}\right) =a\dfrac{d_{i}}{Q_{i}}+h_{i}\dfrac{Q_{i}}{2}
\end{equation*}%
with the optimal order size being now  
\begin{equation*}
\widehat{Q}_{i}=\sqrt{\frac{2a}{h_{i}}\cdot d_{i}}=\sqrt{\frac{2a}{h_{i}}%
\cdot \left[ \underline{d_{i}},\overline{d_{i}}\right] }.
\end{equation*}%
\noindent Thus, the optimal average inventory cost per unit time is $%
c(Q_{i})=2am_{i}$ where 
\begin{equation*}
m_{i}=\dfrac{d_{i}}{\widehat{Q}_{i}}=\left[ \frac{\underline{d_{i}}}{\underline{Q_{i}}}
,\frac{\overline{d_{i}}}{\overline{Q_{i}}}\right]
\end{equation*}
is the optimal number of orders per unit time.

\noindent An interval inventory situation is a triplet $I=\left( N,a,\left\{
m_{i}\right\} _{i\in N}\right) ,$ where $N$ is the
group of agents, $a$ is the fixed ordering cost and $m=(m_{1},\ldots ,m_{n})$ represents the optimal number of orders for the agents. In this case, each $m_{i}=\left[ \underline{m_{i}},
\overline{m_{i}}\right] \in \mathcal{I}(\mathbb{R}^+)$ for all $i\in N.$

\noindent The interval inventory game arising from an interval inventory
situation $I$ is denoted by $\left( N,w^I\right) $. The value $w^I(S)$ represents here the minimal average
inventory cost for coalition $S\subseteq N$, defined as  
\begin{equation*}
w^I(S):=2a\sqrt{\sum_{i\in S}m_{i}^{2}}=2a\sqrt{\sum_{i\in S}\left[ 
\underline{m_{i}^{2}},\overline{m_{i}^{2}}\right] }=\left[ \underline{w}(S),%
\overline{w}(S)\right],
\end{equation*}

where $\underline{w}(S)=2a\sqrt{\sum_{i\in S}\underline{m_{i}^{2}}}$ and $\overline{w}(S)=2a\sqrt{\sum_{i\in S}\overline{m_{i}^{2}}}$.



The next proposition shows that the interval inventory games are concave. Consequently they are balanced and it makes sense for the grand coalition to form because it is the one that generates the minimum cost for all agents.

\begin{prop}
\noindent Every interval inventory game $\left(
N,w^I\right) $ is concave.
\end{prop}

\begin{proof}
\noindent To prove that $\left(
N,w^I\right) $ is concave, we need to verify that, for any agent $i\in N$ and any pair of coalitions $S,T\in 2^{N}$ such that $S\subset T\subset
N\setminus \{i\},$  $w^I(S\cup \{i\})-w^I(S)\succeq w^I(T\cup \{i\})-w^I(T)$ is satisfied. 


Indeed, take $i\in N$ and $S,T\in 2^{N}$ such that $S\subset T\subset N\setminus \{i\}.$ We know that 
\begin{eqnarray*}
w^I(S\cup \{i\}) &=&2a\sqrt{\sum_{j\in S\cup \{i\}}m_{j}^{2}}=2a\sqrt{%
\sum_{j\in S}m_{j}^{2}+m_{i}^{2}}, \\
w^I(T\cup \{i\}) &=&2a\sqrt{\sum_{j\in T\cup \{i\}}m_{j}^{2}}=2a\sqrt{%
\sum_{j\in T}m_{j}^{2}+m_{i}^{2}}.
\end{eqnarray*}

\noindent and 

\begin{equation*}
\sum_{j\in S}m_{j}^{2}\preceq \sum_{j\in T}m_{j}^{2}.
\end{equation*}

\noindent Applying then the concavity property of the square root function, we have  

\begin{equation*}
2a\left( \sqrt{\sum_{j\in S}m_{j}^{2}+m_{i}^{2}}-\sqrt{\sum_{j\in S}m_{j}^{2}%
}\right) \succeq 2a\left( \sqrt{\sum_{j\in T}m_{j}^{2}+m_{i}^{2}}-\sqrt{%
\sum_{j\in T}m_{j}^{2}}\right)
\end{equation*}%
\noindent which simplifies to 

\begin{equation*}
w^I(S\cup \{i\})-w^I(S)\succeq w^I(T\cup \{i\})-w^I(T)
\end{equation*}%
\noindent and the inequality holds.  
\end{proof}

The concavity of interval inventory games ensures that the minimum total cost is attained when all agents cooperate in the grand coalition, i.e., 
$w^I(N)=2a\sqrt{\sum_{i\in N}m_{i}^{2}}$. Accordingly, we examine how this cost can be fairly allocated among the participants. In particular, we analyze the distribution of this minimum total cost using the SOC-rule adapted to interval inventory settings. Moreover, we provide a formal characterization of the interval SOC-rule within a comprehensive framework of interval inventory situations, enhancing the understanding of its function and practical relevance.

\subsection{The interval SOC-rule}

We denote by $\mathcal{IS}^{N}$ the family of interval inventory situations
with a set of agents $N$. Let $\Psi :\mathcal{IS}^{N}\rightarrow \mathcal{I}(%
\mathbb{R})^{N}$ be an allocation rule  for interval inventory
situations such that $\sum_{i\in N}\Psi _{i}(N,a,\left\{ m_{i}\right\} _{i\in N})=2a\sqrt{%
\sum_{i\in N}m_{i}^{2}}.$ We call it \textit{interval allocation rule}. 




\medskip

\noindent A desirable first property for an interval allocation rule is that it should be stable in the sense that no group of agents have incentives to leave the grand coalition without increasing the cost for doing so. It is called \textit{Cross-Coalition Acceptable property.}

\noindent \textbf{Cross-Coalition Acceptable (CCA).} We say that a rule $\Psi $
satisfies the Cross-Coalition Acceptable property if, for all $(N,a,\{ m_i\}
_{i\in N})$ and all $S\subseteq N$  
\begin{equation*}
\begin{split}
\sum_{j\in S}\Psi_j\left(N,a,\left\{ m_i\right\} _{i\in
N}\right)\preceq 2a\sqrt{%
\sum_{j\in S}m_{j}^{2}}.
\end{split}%
\end{equation*}%

We now consider two more valuable properties for such interval allocation rules. The second property states that an agent who does not have an order pays nothing. It is called the \textit{Inactive Agent Exemption property.}



\noindent \textbf{Inactive Agent Exemption (IAE).} We say that a rule $\Psi $
satisfies the Inactive Agent Exemption property if, for all $(N,a,\left\{ m_{i}\right\}
_{i\in N})$ and all $j\in N$ with $m_{j}=\left[ 0,0\right] $, it holds that 
\begin{equation*}
\Psi _{j}(N,a,\left\{ m_{i}\right\} _{i\in N})=\left[ 0,0\right] .
\end{equation*}

\noindent The third property is a special additivity for interval inventory situations. It is called the \textit{Transfer-Based Additivity property }

\noindent \textbf{Transfer-Based Additivity (TBA).} We say that a rule $\Psi $
satisfies the Transfer-Based Additivity property if, for all $(N,a,\left\{ m_{i}\right\}
_{i\in N})$, $(N,a,\left\{ \hat{m}_{i}\right\} _{i\in N})$ and $j\in N$  
\begin{equation*}
\begin{split}
\sqrt{\sum_{i\in N}\left(m_i^2+\hat{m}_i^2\right)}\cdot \Psi_j (N,a,\left\{ 
\sqrt{m^2_{i}+\hat{m}^2_{i}}\right\} _{i\in N})= \\
\sqrt{\sum_{i\in N}m_i^2}\cdot \Psi_j (N,a,\left\{ m_{i}\right\} _{i\in N})+%
\sqrt{\sum_{i\in N}\hat{m}_i^2}\cdot \Psi_j (N,a,\left\{ \hat{m}_{i}\right\}
_{i\in N}).
\end{split}%
\end{equation*}%

\medskip

We now introduce the SOC-rule for interval inventory situations, and we call \textit{the interval SOC-rule}. In fact, the interval SOC-rule $\Gamma (N,a,\left\{ m_{i}\right\} _{i\in
N}) $ is defined for all $j\in N$ as follows 

\begin{equation}  \label{soc.int}
\Gamma_j(N,a,\left\{ m_{i}\right\} _{i\in N})=\frac{2am_j^2}{\sqrt{%
\sum_{i\in N}m_{i}^{2}}}.
\end{equation}

\noindent You might notice that $(\ref{soc.int})$ is well-defined for
interval inventory situations $(N,a,\left\{ m_{i}\right\} _{i\in N})$ that
satisfy 
\begin{equation}  \label{cond.1.}
\underset{i\in N}{\max }\left\{ \underline{m_{i}^{2}}\cdot\frac{1}{\overline{%
m_{i}}^{2}}\right\} \leq \underline{m_{N}}\cdot\frac{1}{\overline{m_{N}}}.
\end{equation}
where $m_N=\sqrt{\sum_{i\in N}m_{i}^{2}}.$ Additionally, it is immediate to see that $\sum_{j\in N}\Gamma_j (N,a,\left\{ m_{i}\right\} _{i\in
N})=2a\sqrt{\sum_{i\in N}m_{i}^{2}}.$

\noindent We focus on interval inventory situations that satisfy the
condition $(\ref{cond.1.})$. We will refer to it simply as interval
inventory situations from now on.

The following example illustrates how we can easily calculate the interval SOC-rule.

\begin{ex} 
Let $N=\{1,2,3\}$, $a=1$, $m_{1}=[1,2]$, $m_{2}=[2,3]$, $m_{3}=[3,4]$. It is easy to check that this interval inventory situation satisfies condition (\ref{cond.1.}). Then, the total cost to allocate in this interval situation is $\left[2\sqrt{14},2\sqrt{29} \right]$,  and the interval SOC-rule is given by

\begin{equation*}
\left( \left[ \frac{2}{\sqrt{14}},\frac{8}{\sqrt{29}}\right] ,\left[ \frac{8}{\sqrt{14}},\frac{18}{%
\sqrt{29}}\right] ,\left[ \frac{18}{\sqrt{14}},\frac{32}{\sqrt{29}}\right]
\right).
\end{equation*}
\end{ex}
\medskip

Next proposition shows that the interval SOC-rule is always acceptable for all the agents in interval inventory situations. 

\begin{prop} Consider an interval inventory situation $(N,a,\left\{ m_{i}\right\} _{i\in N})$. The interval SOC-rule is Cross-Coalition Acceptable.  
\end{prop}

\begin{proof}  
Take $S\subseteq N$. We have to prove that $\sum_{j\in S}\Gamma_j\left(N,a,\left\{ m_l\right\} _{l\in
N}\right)\preceq 2a\sqrt{\sum_{j\in S}m_{j}^{2}}$. 

Indeed, we know that
\begin{equation*}
\sum_{j\in S}\Gamma_j\left(N,a,\left\{ m_l\right\} _{l\in
N}\right)=\frac{2a\sum_{j\in S}
m_{j}^{2}}{\sqrt{\sum_{j\in N}m_{j}^{2}}}.
\end{equation*}

We now prove that 
\begin{equation*}
\frac{\sum_{j\in S}
m_{j}^{2}}{\sqrt{\sum_{j\in N}m_{j}^{2}}} \preceq \sqrt{\sum_{j\in S}m_{j}^{2}}.
\end{equation*}

Indeed, the following chain of inequalities leads us to it 
\begin{equation*}
\frac{\left(\sum_{j\in S}
m_{j}^{2}\right)^2}{\sum_{j\in N}m_{j}^{2}} \preceq \sum_{j\in S}m_{j}^{2}\Leftrightarrow\left(\sum_{j\in S}
m_{j}^{2}\right)^2 \preceq \sum_{j\in S}m_{j}^{2}\sum_{j\in N}m_{j}^{2}\Leftrightarrow \sum_{j\in S}
m_{j}^{2} \preceq\sum_{j\in N}m_{j}^{2}
\end{equation*}%

Then, the SOC-rule satisfies CCA property. 
\end{proof}

\noindent We complete our study of the interval SOC-rule by showing that the above IAE and TBA properties characterize it. 



\begin{thm}
\label{th.1.}  There exists a unique interval allocation rule on the family of interval inventory situations satisfying the properties of IAE and TBA. It is the interval SOC-rule.
\end{thm}

\begin{proof}
First, we prove that the interval SOC-rule satisfies the properties of IAE and TBA.




Indeed, take $i\in N$ such that $m_{i}=\left[ 0,0\right]$, then 
\begin{equation*}
\Gamma_i (N,a,\left\{ m_{j}\right\} _{j\in
N})=\frac{2am_{i}^{2}}{\sqrt{\sum_{j\in N}m_{j}^{2}}}=\left[ 0,0\right],
\end{equation*}

hence $\Gamma$ satisfies IAE. 

Take now $(N,a,\left\{ m_{l}\right\}_{l\in N})$ and $(N,a,\left\{ \hat{m}_{l}\right\} _{l\in N})$ and $i\in N$. Firstly, we know that 
\begin{equation*}
\begin{split}
\sqrt{\sum_{j\in N}(m_j^2+\hat{m}_j^2})\cdot \Gamma_i (N,a,\left\{ \sqrt{m^2_{j}+\hat{m}^2_{j}}\right\} _{j\in
N})=\\
\sqrt{\sum_{j\in N}(m_j^2+\hat{m}_j^2})\cdot\frac{2a\left(\sqrt{m^2_{i}+\hat{m}^2_{i}}\right)^{2}}{\sqrt{\sum_{j\in N}\left(\sqrt{m^2_{j}+\hat{m}^2_{j}}\right)^{2}}}=\\
2a\left(m^2_{i}+\hat{m}^2_{i}\right).\\
\end{split}
\end{equation*}%
Secondly, we have that  
\begin{equation*}
\begin{split}
\sqrt{\sum_{j\in N}m_j^2}\cdot \Gamma_i
(N,a,\left\{ m_{j}\right\} _{j\in
N})+\sqrt{\sum_{j\in N}\hat{m}_j^2}\cdot \Gamma_i (N,a,\left\{ \hat{m}_{j}\right\} _{j\in
N})=\\
\sqrt{\sum_{j\in N}m_j^2}\cdot \frac{2am_i^2}{\sqrt{\sum_{j\in N}m_{j}^{2}}}+\sqrt{\sum_{j\in N}\hat{m}_j^2}\cdot \frac{2a\hat{m}_i^2}{\sqrt{\sum_{j\in N}\hat{m}_{j}^{2}}}=
2a\left(m^2_{i}+\hat{m}^2_{i}\right).
\end{split}
\end{equation*}%

We can then conclude that $\Gamma$ also satisfies TBA property. Now, we show that it is the unique rule that satisfies those two properties. 

Consider $\Psi $ an interval allocation rule that satisfies IAE and TBA. Take an interval inventory situation $(N,a,\left\{ m_{j}\right\} _{j\in
N})$. For each $i\in N$, we can define $(N,a,\left\{ m^i_j\right\} _{j\in
N})$ such that $m^i_i=m_i$ and $m^i_j=\left[0,0\right]$ for all $j\in N\backslash\{i\}$. By the TBA we have for all $i\in N$ 

\begin{equation*}
\begin{split}
    \sqrt{\sum_{j\in N}\sum_{l\in N}(m^j_l)^2}\cdot \Psi_i \left(N,a,\left\{ \sqrt{\sum_{j\in N}(m^j_l)^2}\right\} _{l\in
N}\right)=
\sum_{j\in N}\sqrt{\sum_{l\in N}(m^j_l)^2}\cdot\Psi_i \left(N,a,\left\{ m^j_l\right\} _{l\in
N}\right)\Leftrightarrow\\
\sqrt{\sum_{l\in N}m_l^2}\cdot \Psi_i \left(N,a,\left\{ m_l\right\} _{l\in
N}\right)=\sum_{j\in N}m_j\cdot\Psi_i \left(N,a,\left\{ m^j_l\right\} _{l\in
N}\right).
\end{split}
\end{equation*}
\noindent Now, taking into account that $\Psi $ is  an interval allocation rule and satisfies IAE, we have that 
\begin{equation*}
\Psi_i \left(N,a,\left\{ m^j_l\right\} _{l\in
N}\right)=%
\begin{cases}
2am_i & \text{if }i=j, \\ 
\lbrack 0,0] & \text{if }i\neq j,%
\end{cases}%
\end{equation*}%
and, then 
\begin{equation*}
\sqrt{\sum_{l\in N}m_l^2}\cdot \Psi_i \left(N,a,\left\{ m_l\right\} _{l\in
N}\right)=2am_i^2
\end{equation*}%
so $\Psi _{i}\left(N,a,\left\{ m_l\right\} _{l\in
N}\right)=\Gamma _{i}\left(N,a,\left\{ m_l\right\} _{l\in
N}\right)$. 
\end{proof}
\medskip

Next, we show the independence of the IAE and TBA properties in Theorem \ref{th.1.}.

 \begin{enumerate}
     \item \textbf{NO TBA.} The following interval allocation rule satisfies IAE but not TBA
     \begin{equation*}
\Psi_j \left(N,a,\left\{m_i\right\} _{i\in
N}\right)=%
\begin{cases}
[0,0] & \text{if }  j\in IA\left(N,a,\left\{m_i\right\} _{i\in
N}\right) \\ 
\frac{2a\sqrt{%
\sum_{i\in N}m_{i}^{2}}}{|N|-|IA\left(N,a,\left\{m_i\right\} _{i\in
N}\right)|}& \text{otherwise. }
\end{cases}%
\end{equation*}%
for all $j\in N$, where $IA\left(N,a,\left\{m_i\right\} _{i\in
N}\right)=\{j\in N\;|\; m_j=[0,0]\}$.
     \item \textbf{NO IAE.} The following interval allocation rule satisfies TBA but not IAE 
\begin{equation*}
\Psi \left(N,a,\left\{m_i\right\} _{i\in
N}\right)=%
\begin{cases}
(\Gamma_{2}(N,a,\left\{ m_{i}\right\}_{i\in
N}),\Gamma_{1}(N,a,\left\{ m_{i}\right\}_{i\in
N})) & \text{if } \mid N\mid =2 \\ 
\Gamma (N,a,\left\{ m_{i}\right\}_{i\in
N}) & \text{if } \mid N \mid \neq 2.%
\end{cases}%
\end{equation*}%
 \end{enumerate}

In the following, we focus on examining the interval Shapley rule, as introduced in \cite{Alparslan2010Axiomatization}, within the framework of interval inventory situations. Our study aims to explore its properties and implications in this context. Furthermore, we propose a characterization of the interval Shapley rule within the broader family of interval inventory situations, providing a deeper understanding of its role and applicability. 

\section{The interval Shapley rule}\label{sec:4}

\noindent Consider an interval inventory situation $\left( N,a,\left\{
m_{i}\right\} _{i\in N}\right)$. We define the \textit{interval Shapley rule} for this interval inventory situation as the Shapley value of the associated interval inventory game $\left( N,w^I\right).$ Formally  

\begin{equation}  \label{shap}
Sh\left( N,a,\left\{ m_{i}\right\} _{i\in N}\right) :=\varphi(N,w^I)
\end{equation}%
where $\varphi $ corresponds to interval Shapley value as described in \ref{CIG}. We know the interval Shapley value is $\varphi_j(N,w^I)=\left[ \phi_j(N,\underline{w^I}),\phi_j
(N,\overline{w^I})\right] $ for all $j \in N$, and $\phi_j$ is the Shapley value for
the inventory games $(N,\underline{w^I})$ and $(N,\overline{w^I}).$ 
\smallskip

\noindent The interval Shapley rule is well-defined for interval inventory situations $(N,a,\left\{ m_{i}\right\} _{i\in N})$ satisfying $|m_{S}|\geq |m_{T}|$, for all $S\subseteq T\subseteq N\setminus\{i\}$ and  $i\in N$, with $m_S=\sqrt{\sum_{i\in S}m_i^2}$. Moreover, it is an interval allocation rule
because of $\sum_{i\in N} Sh_{j}(N,a,\left\{ m_{i}\right\} _{i\in N})=2a\sqrt{\sum_{i\in N}m_{i}^{2}}$. We focus on those situations which we call \textit{ monotonic interval inventory situations} and we aim to characterize the Shapley rule. 


Before doing so, we should note that interval Shapley rule on monotonic interval inventory situations always satisfies the Cross-Coalition Acceptable property. Indeed, since $(N,w^{I})$ is a concave interval cost game then 

$$\sum_{j\in S}Sh_{j}\left(N,a,\left\{ m_{i}\right\} _{i\in
N}\right) = \sum_{j\in S}\varphi _{j}(N,w^{I}%
)\preceq w^{I}(S)=2a\sqrt{%
\sum_{i\in S}m_{i}^{2}}.$$

The following example illustrates the behavior of the interval Shapley rule in monotonic interval inventory situations.

\begin{ex}
Let $N=\{1,2,3\}$, $a=1$, $m_{1}=[1,3]$, $m_{2}=[2,4]$, $m_{3}=[3,5]$. It is easy to check that this is a monotonic interval inventory situation and the total cost to allocate is $[7.48,14.14].$  Then, the interval Shapley rule is given by 


\begin{equation*}
Sh\left(N,a,\left\{ m_{i}\right\} _{i\in N}\right)=\left(
[0.89,3.06],[2.33,4.63],[4.26,6.46]\right) .
\end{equation*}
\end{ex}

To complete the analysis of the interval Shapley rule in monotonic interval inventory situations, we introduce a new desirable property for interval allocation rules. It states that in the context of an interval inventory problem, the leaving of one of two agents results in a change to both the minimum and maximum cost for the remaining agent that is identical to the change that would occur were the roles reversed. It is called \textit{Balanced Cost}.


\textbf{Balanced Cost (BC).} We say that an interval allocation rule $\Psi$ satisfies balanced cost if, for any $(N,a,\{m_{i}\}_{i \in N})$ and for any $i,j \in N$
\begin{equation*}
\underline{\Psi_{i}(N,a,\left\{ m_{l}\right\} _{l\in N})} -\underline{%
\Psi_{i}(N\backslash\{j\},a,\left\{ m_{il}\right\} _{l\in N\backslash\{j\}})}
=\underline{\Psi_{j}(N,a,\left\{ m_{l}\right\} _{il\in N})} -\underline{%
\Psi_{j}(N\backslash\{i\},a,\left\{ m_{il}\right\} _{l\in N\backslash\{i\}})}
\end{equation*}%
and 
\begin{equation*}
\overline{\Psi_{i}(N,a,\left\{ m_{l}\right\} _{il\in N})} -\overline{%
\Psi_{i}(N\backslash\{j\},a,\left\{ m_{l}\right\} _{il\in N\backslash\{j\}})}
=\overline{\Psi_{j}(N,a,\left\{ m_{l}\right\} _{il\in N})} -\overline{%
\Psi_{j}(N\backslash\{i\},a,\left\{ m_{l}\right\} _{l\in N\backslash\{i\}})}.
\end{equation*}

The following theorem demonstrates that the interval Shapley rule is the only interval allocation rule that satisfies the BC property for monotonic interval inventory situations.

\begin{thm}
\label{th.2.} The Shapley rule is the unique rule for monotonic interval inventory
situations that satisfies BC property.
\end{thm}

\begin{proof}
We begin by proving that Shapley rule satisfies Balanced Cost property. Consider $\left( N,a,\left\{ m_{l}\right\} _{l\in N}\right)$. Take $i,j\in N$ with $i\neq j$. Let us prove the first condition of the property; the second one is proved in a similar way. Indeed, 
\begin{eqnarray*}
\underline{Sh_{i}\left( N,a,\left\{ m_{l}\right\} _{l\in N}\right)} -\underline{Sh_{i}\left( N\backslash\{j\},a,\left\{ m_{l}\right\} _{l\in N\backslash\{j\}}\right)} &=&\underline{\varphi
_{i}(N,w^I)}-\underline{\varphi _{i}(N\backslash\{j\} ,w^{I^{-j}})} 
\\
\underline{Sh_{j}\left( N,a,\left\{ m_{l}\right\} _{l\in N}\right)} -\underline{Sh_{j}\left( N\backslash\{i\},a,\left\{ m_{l}\right\} _{l\in N\backslash\{i\}}\right)} &=&\underline{\varphi
_{j}(N,w^{I})}-\underline{\varphi _{j}(N\backslash\{i\} ,w^{I^{-i}})} 
\end{eqnarray*}
where $I^{-j}=\left( N\backslash\{j\},a,\left\{ m_{l}\right\} _{l\in N\backslash\{i\}}\right)$ for all $j \in N$. Notice that $\underline{w^{I^{-j}}}$ is a cooperative interval game with set
of players $N\setminus \left\{ j\right\}$ such that, for all $S\subseteq
N\backslash\{j\}, $
\begin{equation*}
\underline{w^{I^{-j}}}\left( S\right) =2a\sqrt{\sum\nolimits_{i\in S}m_{i}^{2}}%
=\underline{w^{I}}\left( S\right).
\end{equation*}%

We know that 
\begin{eqnarray*}
\varphi _{i}(N,w^{I})=
\left[ \phi_{i} (N,\underline{w^{I}}),\phi_{i} (N,\overline{w^{I}})\right]
\end{eqnarray*}%
then
\begin{eqnarray*}
\underline{\varphi _{i}(N,w^{I})}-\underline{\varphi _{i}(N\backslash{j},w^{I^{-j}})} = 
 \phi_{i} (N,\underline{w^{I}})
- \phi_{i} (N\backslash\{j\} ,\underline{w^{I^{-j}}})
\end{eqnarray*}%
and
\begin{eqnarray*}
\underline{\varphi _{j}(N,w^{I})}-\underline{\varphi _{j}(N\backslash{i},w^{I^{-i}})} = \phi_{j} (N,\underline{w^{I}})
- \phi_{j} (N\backslash\{i\} ,\underline{w^{I^{-i}}})
\end{eqnarray*}%


By the other way, for all $S\subseteq N\backslash\{j\}$, $%
\underline{w^{I}}\left( S\right) =\underline{w_{-j}^{I}}\left( S\right) $ where $%
w_{-j}^{I}\left( S\right) $ denotes restriction of the cooperative
interval game $\underline{w}^{I}$ to $N\setminus \left\{ j\right\} .$ Then,%
\begin{equation}
\label{eq.1.}
\underline{w}^{I^{-j}}\left( S\right) =\underline{w_{-j}}^{I}\left( S\right)
\end{equation}%

Given that \cite {Myerson1980} established that the Shapley value satisfies the Balanced Contribution property\footnote{This property states that for each TU-game $(N,v)$ and each $i,j\in N$ with $i\neq j$ then $\phi_i(N,v)-\phi_i(N\backslash\{j\},v_{-j})=\phi_j(N,v)-\phi_j(N\backslash\{i\},v_{-i})$.}, it follows that

\begin{eqnarray*}
\underline{Sh_{i}\left( N,a,\left\{ m_{l}\right\} _{l\in N}\right)} -\underline{Sh_{i}\left( N\backslash\{j\},a,\left\{ m_{l}\right\} _{l\in N\backslash\{j\}}\right)} =&\\
\underline{\varphi
_{i}(N,w^I)}-\underline{\varphi _{i}(N\backslash\{j\} ,w^{I^{-j}})} =&
\\
 \phi_{i} (N,\underline{w^{I}})
- \phi_{i} (N\backslash\{j\} ,\underline{w^{I^{-j}}}) =&\\
 \phi_{i} (N,\underline{w^{I}})
- \phi_{i} (N\backslash\{j\} ,\underline{w^{I}_{-j}})=& \\
 \phi_{j} (N,\underline{w^{I}})
- \phi_{j} (N\backslash\{i\} ,\underline{w^{I}_{-i}}) =&\\
 \underline{\varphi
_{j}(N,w^{I})}-\underline{\varphi _{j}(N\backslash\{i\} ,w^{I^{-i}})} =&\\
 \underline{Sh_{j}\left( N,a,\left\{ m_{l}\right\} _{l\in N}\right)} -\underline{Sh_{j}\left( N\backslash\{i\},a,\left\{ m_{l}\right\} _{l\in N\backslash\{i\}}\right)}.  
\end{eqnarray*}

\noindent Subsequently, we demonstrate that the interval Shapley rule is the unique interval allocation rule that satisfies BC property. Assume the existence of an alternative interval allocation rule for interval inventory situations that satisfies the BC property, denoted $R\left( N,a,\left\{ m_{l}\right\} _{l\in N}\right)$. We establish, via induction on the number of agents, that this rule coincides with the interval Shapley rule.

\noindent  If $\left\vert N\right\vert =1$ by definition of allocation rule
\begin{equation*}
R\left( \left\{ i\right\},a, m_{i}\right) =2am_{i}^{2}=Sh\left( \left\{ i\right\},a, m_{i}\right).
\end{equation*}%
If $\left\vert N\right\vert \geq 2.$ Assume that $\left( N,a,\left\{ m_{l}\right\} _{l\in N}\right)$ is minimal in the sense
that 
$$R\left( N\backslash\{j\},a,\left\{ m_{l}\right\} _{l\in N\backslash\{j\}}\right) =Sh\left( N\backslash\{j\},a,\left\{ m_{l}\right\} _{l\in N\backslash\{j\}}\right) $$
for any $j\in N$. Consider now $i,j\in N$ with $i\neq j$, then as $R\left( N,a,\left\{ m_{l}\right\} _{l\in N}\right) $ and $Sh\left( N,a,\left\{ m_{l}\right\} _{l\in N}\right) $ satisfies BC property, we have that
\begin{eqnarray*}
\overline{R_{i}(N,a,\left\{ m_{i}\right\} _{i\in
N})} -\overline{R_{i}(N\backslash\{j\},a,\left\{ m_{i}\right\} _{i\in
N\backslash\{j\}})}& = \\
\overline{R_{j}(N,a,\left\{ m_{i}\right\} _{i\in
N})}
-\overline{R_{j}(N\backslash\{i\},a,\left\{ m_{i}\right\} _{i\in
N\backslash\{i\}})},&\\
\overline{Sh_{i}(N,a,\left\{ m_{i}\right\} _{i\in
N})} -\overline{Sh_{i}(N\backslash\{j\},a,\left\{ m_{i}\right\} _{i\in
N\backslash\{j\}})} &=\\
\overline{Sh_{j}(N,a,\left\{ m_{i}\right\} _{i\in
N})}
-\overline{Sh_{j}(N\backslash\{i\},a,\left\{ m_{i}\right\} _{i\in
N\backslash\{i\}})},
\end{eqnarray*}%
and
\begin{eqnarray*}
\underline{R_{i}(N,a,\left\{ m_{i}\right\} _{i\in
N})} -\underline{R_{i}(N\backslash\{j\},a,\left\{ m_{i}\right\} _{i\in
N\backslash\{j\}})} &=\\
\underline{R_{j}(N,a,\left\{ m_{i}\right\} _{i\in
N})}
-\underline{R_{j}(N\backslash\{i\},a,\left\{ m_{i}\right\} _{i\in
N\backslash\{i\}})},&\\
\underline{Sh_{i}(N,a,\left\{ m_{i}\right\} _{i\in
N})} -\underline{Sh_{i}(N\backslash\{j\},a,\left\{ m_{i}\right\} _{i\in
N\backslash\{j\}})} &=\\
\underline{Sh_{j}(N,a,\left\{ m_{i}\right\} _{i\in
N})}
-\underline{Sh_{j}(N\backslash\{i\},a,\left\{ m_{i}\right\} _{i\in
N\backslash\{i\}})}
\end{eqnarray*}%
and by the minimality of $\left( N\backslash\{j\},a,\left\{ m_{l}\right\} _{l\in N\backslash\{j\}}\right)$;
\begin{eqnarray*}
\overline{R_{i}(N,a,\left\{ m_{i}\right\} _{i\in
N})} -
\overline{R_{j}(N,a,\left\{ m_{i}\right\} _{i\in
N})} &=\\
\overline{Sh_{i}(N,a,\left\{ m_{i}\right\} _{i\in
N})} -
\overline{Sh_{j}(N,a,\left\{ m_{i}\right\} _{i\in
N})},
\end{eqnarray*}%
and
\begin{eqnarray*}
\underline{R_{i}(N,a,\left\{ m_{i}\right\} _{i\in
N})} -
\underline{R_{j}(N,a,\left\{ m_{i}\right\} _{i\in
N})} &=\\
\underline{Sh_{i}(N,a,\left\{ m_{i}\right\} _{i\in
N})} -
\underline{Sh_{j}(N,a,\left\{ m_{i}\right\} _{i\in
N})}.&
\end{eqnarray*}%

Therefore, summing over $i\in N$
\begin{eqnarray*}
\sum_{i\in N}\left(\overline{R_{i}(N,a,\left\{ m_{i}\right\} _{i\in
N})} - \overline{Sh_{i}(N,a,\left\{ m_{i}\right\} _{i\in
N})}\right)
 &=\\
 |N|\left(\overline{R_{j}(N,a,\left\{ m_{i}\right\} _{i\in
N})} -
\overline{Sh_{j}(N,a,\left\{ m_{i}\right\} _{i\in
N})}\right),
\end{eqnarray*}%
and
\begin{eqnarray*}
\sum_{i\in N}\left(\underline{R_{i}(N,a,\left\{ m_{i}\right\} _{i\in
N})} - \underline{Sh_{i}(N,a,\left\{ m_{i}\right\} _{i\in
N})}\right)
 &=\\
 |N|\left(\underline{R_{j}(N,a,\left\{ m_{i}\right\} _{i\in
N})} -
\underline{Sh_{j}(N,a,\left\{ m_{i}\right\} _{i\in
N})}\right).&
\end{eqnarray*}%
Now, by definition of allocation rule, we have that
\begin{eqnarray*}
0
 &=
 |N|\left(\overline{R_{j}(N,a,\left\{ m_{i}\right\} _{i\in
N})} -
\overline{Sh_{j}(N,a,\left\{ m_{i}\right\} _{i\in
N})}\right),&\\
0
 &=
 |N|\left(\underline{R_{j}(N,a,\left\{ m_{i}\right\} _{i\in
N})} -
\underline{Sh_{j}(N,a,\left\{ m_{i}\right\} _{i\in
N})}\right).&
\end{eqnarray*}%
and then 
\begin{equation*}
R_{j}(N,a,\left\{ m_{i}\right\} _{i\in
N}) =Sh_{j}(N,a,\left\{ m_{i}\right\} _{i\in
N})
\end{equation*}
for all $j\in N$.
\end{proof}

We observe that the interval SOC-rule is an interval allocation rule that fails to satisfy the Balanced Contribution (BC) property. 

We complete our study of cost allocations in interval inventory situation by developing a comparative analysis of the the interval individual cost, the interval Shapley rule, the interval SOC-rule  through a case study, emphasizing their differences and exploring potential connections.

\section{Perfume Inventory Situation for Duty-Free Stores at Major Spanish Airports}\label{sec:5}

This study focuses on the main Spanish airports with the objective of optimizing the joint management of perfume inventories in duty-free retail outlets located at these sites. Due to the inherent uncertainty in perfume demand, although within known bounds, the problem is appropriately modeled as an interval inventory situation. This modeling approach captures demand variability while facilitating robust, coordinated inventory decisions across multiple locations and enabling a fair allocation of the total cost resulting from cooperative management.

The analysis specifically considers the airports of Madrid-Barajas, Barcelona-El Prat, Málaga-Costa del Sol, Valencia, Alicante, Palma de Mallorca, and Sevilla. Using an interval-based EOQ model, the study provides inventory optimization under uncertainty, taking into account fluctuations in demand. Demand intervals are derived from AENA 2023 passenger traffic statistics (Madrid: 60200000,
Barcelona: 49900000, Malaga: 22300000, Valencia: 10500000, Alicante:
15700000, Mallorca: 31200000, Sevilla: 8100000; source: \texttt{www.aena.es}~%
\cite{aena2023}). 

The product considered is perfume in 50 ml bottles, priced
at 50 euro per unit. Perfume demand assumes a 5\% annual purchase rate by
passengers, supported by a 12\% duty-free sales growth~\cite{Reuters2024}.
Demand intervals $d_{i}=[\underline{d}_{i},\overline{d}_{i}]$ reflect a $\pm
30\%$ seasonal variation. 

The interval inventory situation is defined for $N=\{1,2,3,4,5,6,7\}$ with Madrid(1), Barcelona(2), Mallorca(3), Malaga(4), Alicante(5), Valencia(6), Sevilla(7), with monthly demand intervals and holding costs as
follows

\begin{itemize}
\item Madrid-Barajas: $d_1 = [175 000, 325 000]$, $h_1 = 10$ euro/unit/year,

\item Barcelona-El Prat: $d_2 = [145 600, 270 400]$, $h_2 = 12$
euro/unit/year,

\item Mallorca: $d_3 = [90 900, 168 900]$, $h_3 = 10$ euro/unit/year,

\item Malaga-Costa del Sol: $d_4 = [64 400, 119 600]$, $h_4 = 8$
euro/unit/year,

\item Alicante: $d_5 = [45 700, 84 900]$, $h_5 = 11$ euro/unit/year,

\item Valencia: $d_6 = [30 600, 56 700]$, $h_6 = 9$ euro/unit/year,

\item Sevilla: $d_7 = [23 600, 43 800]$, $h_7 = 7$ euro/unit/year.
\end{itemize}

 The ordering cost is fixed at $a=200$ euro per order in all stores. Therefore, it is easy to see that the optimal number of orders per unit of time is given by
\begin{itemize}
\item Madrid-Barajas: $m_1 = [66.14, 90.13]$,

\item Barcelona-El Prat: $m_2 = [66.09, 90.06]$,

\item Mallorca: $m_3 = [47.67, 64.99]$,

\item Malaga-Costa del Sol: $m_4 = [35.89, 48.90]$,

\item Alicante: $m_5 = [35.45, 48.32]$,

\item Valencia: $m_6 = [26.23, 35.72]$,

\item Sevilla: $m_7 = [20.32, 27.69]$,
\end{itemize}
and the total cost is equal to $[48434.29,66004.24]$.  It is easy to check that this an monotonic interval inventory situation that satisfies condition (\ref{cond.1.}). 

Table \ref{tabla.1.} below compares the SOC-rule with individual costs. It shows that the former significantly reduces the latter for all airports. The SOC-rule is relatively easy to calculate and reflects the effect of demand on cost sharing. Madrid and Barcelona, with higher demand intervals, receive larger cost shares under the SOC-rule, while smaller airports such as Seville receive smaller shares. Columns three and four report the length of the interval for individual costs ($L_{IC}$) and for the SOC-rule allocation ($L_{SOC}$), respectively. Since the SOC-rule yields shorter intervals, it results in a more precise allocation compared to individual cost intervals.

 \begin{table}[h]
    \centering
    \begin{tabular}{c|c|c|c|c|}
        \textbf{Airports} &\textbf{Individual Costs} &\textbf{SOC-rule} & \textbf{$L_{IC}$} & \textbf{$L_{SOC}$} \\ \hline
         Madrid-Barajas &  [26457.51, 36055.51] & [10757.41, 14660.65] & 9598.00 & 3903.24 \\
        Barcelona-El Prat &  [26436.34, 36026.66] & [10748.80, 14648.92]  & 9590.32&  3900.12 \\
         Mallorca &  [19068.30, 25992.31]  & [7753.01, 10568.82] & 6924.01 & 2815.81 \\
         Malaga-Costa del Sol &  [14355.49, 19563.23] & [5836.83, 7954.67]  & 5207.74 & 2117.84 \\
             Alicante &  [14180.27, 19327.70]  & [5765.58, 7858.90] & 5147.43 & 2093.32  \\
         Valencia &  [10495.71, 14287.06] & [4267.47, 5809.31] & 3791.35 &  1541.84 \\
         Sevilla &  [8128.96, 11074.29]  & [3305.17, 4502.96] & 2945.33 &  1197.79 \\
    \end{tabular}
    \caption{Individual Costs v.s. SOC-rule.}
    \label{tabla.1.}
\end{table}

Table \ref{tabla.2.} below compares the interval Shapley rule with individual costs. The former also significantly
reduces the latter for each airport. Although the interval Shapley rule is much more difficult to compute than the SOC-rule, there are now sampling techniques that allow us to obtain the former for situations with a large number of agents. Again, Madrid and Barcelona, with higher demand intervals, receive larger cost shares under the interval Shapley rule, while smaller airports such as Seville receive smaller shares, reflecting demand-driven allocation. Columns three and four now present the interval lengths for individual costs ($L_{IC}$)
 and for the Shapley rule ($L_{Sh}$), respectively. As in previous cases, the Shapley rule produces shorter intervals, thereby providing a more accurate allocation than that derived from individual cost intervals.

\begin{table}[h]

    \centering
    \begin{tabular}{c|c|c|c|c|}
        \textbf{Airports} &\textbf{Individual Costs} &\textbf{Shapley rule}  & \textbf{$L_{IC}$} & \textbf{$L_{Sh}$} \\ \hline
         Madrid-Barajas &  [26457.51, 36055.51] & [12990.28, 17703.19] & 9598.00& 4712.91\\
         Barcelona-El Prat &  [26436.34, 36026.66]  & [12973.37, 17680.14]  & 9590.32 &  4706.77\\
          Mallorca &  [19068.30, 25992.31]  & [7661.97, 10446.05] & 6924.01 & 2784.08  \\
          Malaga-Costa del Sol &  [14355.49, 19563.23]  & [4896.94, 6673.63] & 5207.74& 1776.69 \\
             Alicante &  [14180.27, 19327.70] & [4804.24, 6549.01] & 5147.43 & 1744.77 \\
          Valencia &  [10495.71, 14287.06] & [3031.00, 4123.68]  & 3791.35 & 1092.68 \\
        Sevilla &  [8128.96, 11074.29]  & [2076.49, 2828.54] & 2945.33 &  752.05  \\
    \end{tabular}
    \caption{Individual Costs v.s. Shapley rule.}
    \label{tabla.2.}
\end{table}

Table \ref{tabla.3.} presents a comparison between the SOC-rule and the interval Shapley rule. We observe that the SOC-rule reduces the costs assigned to high-demand airports, such as Madrid and Barcelona, compared to the Shapley rule, while also increasing precision (i.e., $L_{SOC}<L_{Sh}$). In contrast, for airports with lower demand, such as Málaga, Alicante, Valencia, and Sevilla, the SOC-rule results in higher costs and lower precision ($L_{SOC}<L_{Sh}$). For airports with intermediate demand, such as Mallorca, both the allocated costs and their precision are similar under the two rules. The results show no significant quantitative differences between the two rules, which yield stable and closely aligned allocations. Thus, the choice may depend on the available computational resources and the preferred theoretical properties, whether it satisfies TBA and IAE or adheres to the BC criterion.

 \begin{table}[h]
    \centering
    \begin{tabular}{c|c|c|c|c|}
        \textbf{Airports} &\textbf{SOC-rule} &\textbf{Shapley rule}   & \textbf{$L_{SOC}$} & \textbf{$L_{Sh}$}\\ \hline
         Madrid-Barajas &   [10757.41, 14660.65] & [12990.28, 17703.19] & 3903.24 & 4712.91 \\
        Barcelona-El Prat  & [10748.80, 14648.92] & [12973.37, 17680.14] & 3900.12 & 4706.77 \\
         Mallorca  & [7753.01, 10568.82]  & [7661.97, 10446.05]   &  2815.81 &  2784.08 \\
         Malaga-Costa del Sol  & [5836.83, 7954.67]  & [4896.94, 6673.63]  & 2117.84 & 1776.69 \\
             Alicante  & [5765.58, 7858.90]  & [4804.24, 6549.01]  & 2093.32&  1744.77 \\
         Valencia  & [4267.47, 5809.31]  & [3031.00, 4123.68] & 1541.84 & 1092.68 \\
         Sevilla   & [3305.17, 4502.96]  & [2076.49, 2828.54] & 1197.79 & 752.05 \\
    \end{tabular}
    \caption{SOC-rule v.s. Shapley rule.}
    \label{tabla.3.}
\end{table}

This study case demonstrates the effectiveness of the interval inventory situations with the EOQ model in optimizing inventory management under demand uncertainty and allocating the cost of cooperation for duty-free stores at major Spanish airports.  Using real-world data, such as the AENA 2023 statistics and an estimated 5\% passenger
purchase rate, it provides a quantitative basis for decision making.
The consideration of a $\pm 30\%$ seasonal variability and a 12\% annual
sales growth highlights the model's adaptability to real-market conditions.

\section{Conclusions}\label{sec:6} 

\noindent This paper has examined the application of interval cooperative models to inventory management under uncertainty, addressing the limitations of classical approaches that rely on deterministic assumptions. By extending EOQ-based cooperative inventory games to an interval framework, we offer a more realistic and robust representation of demand variability, an increasingly relevant concern in today’s volatile operational environments.

\noindent Focusing on monotonic interval situations, we developed and analyzed two key cost allocation rules: the interval SOC-rule and the interval Shapley rule. Both were derived using an axiomatic approach, which guarantees their theoretical strength and uniqueness. The interval SOC-rule, which allocates costs proportionally, satisfies the axioms of Inactive Agent Exemption (IAE) and Transfer-Based Additivity (TBA). In contrast, the interval Shapley rule is uniquely characterized by the Balanced Cost (BC) property. Importantly, both rules satisfy the Cross-Coalitional Acceptability (CCA) criterion, ensuring stable and equitable allocations even under uncertainty.

\noindent The practical relevance of these methods was illustrated through a real-world case study involving the joint management of perfume inventories across seven major Spanish airports. Using 2023 AENA passenger traffic data, the analysis demonstrated that the proposed interval-based models lead to more equitable and context-sensitive cost allocations. The comparison between the SOC-rule and the interval Shapley rule revealed minimal quantitative differences, with both rules providing stable and closely aligned allocations. As such, the decision to adopt one rule over the other may depend primarily on the computational capabilities of the agents and the specific axiomatic properties prioritized in the application context.

\noindent Overall, this study contributes to the advancement of cooperative inventory planning under uncertainty by bridging theoretical developments in interval cooperative games with practical, real-world applications. The interval inventory framework introduced here offers a flexible and rigorous tool for designing fair and stable cost-sharing mechanisms in uncertain environments.

\noindent Future research may build on this foundation by exploring alternative representations of uncertainty, such as fuzzy sets or stochastic intervals, incorporating logistical constraints like transportation and storage, or applying these models to other sectors where joint inventory management is essential. Additional directions include extending the current framework to multi-product systems or dynamic demand settings, further enhancing its applicability in complex supply chain contexts.

\section*{Acknowledgments}
This work is part of the R+D+I project grants   PID2021-12403030NB-C31 and PID2022-137211NB-100, that are funded by MCIN/AEI/10.13039/501100011033/ and by ERDF ’’A way of making Europe’’/EU. This work was also partially supported by the grant CIPROM/2024/34 funded by the Valencian Community.

\section*{Statements and Declarations}

\textbf{Conflicts of interest} The authors state that there is no conflict of interest.

\bibliography{referencias}
\bibliographystyle{apalike}

\end{document}